\documentclass[12pt]{amsart}
\usepackage[utf8]{inputenc}
\usepackage{xcolor}
\usepackage[margin=1.5in]{geometry}
\usepackage{indentfirst}
\usepackage{amssymb}

\theoremstyle{plain}
\newtheorem{thm}{Theorem}
\newtheorem{lem}[thm]{Lemma}

\newtheorem{cor}[thm]{Corollary}

\theoremstyle{definition}

\title{Independence of Essential Sets in Finite Implication Bases}
\author{Todd Bichoupan}
\date{April 13, 2023}
\begin{document}

\begin{abstract}
A new characterization is given to describe implication bases of a closure system in terms of the system's quasi-closed sets. Using this characterization, it is possible to show that groups of implications corresponding to distinct essential sets are interchangeable across different bases. It follows from this result that the sum of cardinalities of right sides of all implications corresponding to a single essential set in an optimal basis is fixed, solving an open conjecture by K. Adaricheva and J.B. Nation in 2014. These results provider greater insight into the global structure of implication bases.
\end{abstract}

\maketitle

Definitions and notational conventions are borrowed from \cite{CaMj03}. \newline

A \emph{closure system} $\langle X, \phi \rangle$ is a nonempty set $X$ equipped with a \emph{closure operator} $\phi: \mathcal{P}(X) \mapsto \mathcal{P}(X)$ that satisfies the following for all $A, B \subseteq X$:
\begin{itemize}
    \item $A \subseteq \phi(A)$
    \item $A \subseteq B \implies \phi(A) \subseteq \phi(B)$
    \item $\phi(\phi(A))=\phi(A)$
\end{itemize}
A set $A \subseteq X$ is \emph{closed} if $A = \phi(A)$. $X$ itself is closed, and any intersection of closed sets is closed. The family of closed sets associated with a closure system is unique in the sense that any two distinct closure operators on a set $X$ generate distinct families of closed sets. Moreover, any family of subsets of $X$ that is closed under set intersection and contains $X$ is the family of closed sets associated with some closure operator. \newline

An \emph{implication} on a nonempty set $X$ is an ordered pair of sets $A, B \subseteq X$ denoted $A \rightarrow B$. A set $S \subseteq X$ obeys an implication $A \rightarrow B$ if $A \not \subseteq S$ or $B \subseteq S$. For any set $\Sigma$ of implications on $X$, the family of sets that obey all implications in $\Sigma$ form a closure system, $\langle X, \phi \rangle$, and $\Sigma$ is said to be an \emph{implication basis} for $\langle X, \phi \rangle$. \newline

If $\langle X, \phi \rangle$ is a closure system and $F$ is the associated family of closed sets, then a set $Q \subseteq X$ is \emph{quasi-closed} if $Q \not \in F$ and $F \cup \{ Q \}$ is still closed under set intersection (i.e., for every $S \in F$, $S \supseteq Q$ or $S \cap Q \in F$). For any quasi-closed set $Q$, the closure $\phi(Q)$ is called an \emph{essential set}. If $\mathcal{Q}$ is the family of all quasi-closed sets associated with $\langle X, \phi \rangle$, then $F \cup \mathcal{Q}$ turns out to be closed under set intersection and has an associated closure operator, $\sigma$; $\sigma$ is called the \emph{saturation operator} associated with $\langle X, \phi \rangle$. \newline

A \emph{quasi-closed} set $Q$ associated with a closure system $\langle X, \phi \rangle$ is called a \emph{critical set} if there is no quasi-closed set $S \subsetneq Q$ such that $\phi(S)=\phi(Q)$. In \cite{GuDq86}, J.L. Guigues and V. Duquenne showed that if $X$ is finite, the set of implications $\{C \rightarrow \phi(C): C \text{ is critical}\}$ is an implication basis for $\langle X, \phi \rangle$. Guigues and Duquenne showed further that if $\Sigma$ is an implication basis for $\langle X, \phi \rangle$, where $X$ is finite, and $\sigma$ is the saturation operator associated with $\langle X, \phi \rangle$, then for every critical set $C$ there must be an implication $A \rightarrow B$ in $\Sigma$ such that $\sigma(A)=C$. In particular, $\phi(A)$ is an essential set. \newline

A stronger characterization of the implication bases associated with a closure system is given. In particular, it is possible to more precisely describes the right sides of implications in a basis. This characterization can be used to show that groups of implications corresponding to distinct essential sets are independent, in the sense that those groups of implications can be combined arbitrarily to form a valid basis. This result resolves a Conjecture 67 of \cite{AdN14} about the right sides of \emph{optimal bases} – implication bases where the sum of all cardinalities of the left and right sides of all implications is minimal. \newline

\noindent
The following lemma is equivalent to Proposition 19 of \cite{CaMj03}.
\begin{lem}\label{QuasInc}
Let $X$ be a finite set and let $\langle X, \phi_1 \rangle$ and $\langle X, \phi_2 \rangle$ be two closure systems on $X$. Let $F_1$ be the family of closed sets associated with $\langle X, \phi_1 \rangle$ and let $F_2$ be the family of closed sets associated with $\langle X, \phi_2 \rangle$. Suppose that $F_1 \subsetneq F_2$. Then there exists a a quasi-closed set $Q$ associated with $\langle X, \phi_1 \rangle$ such that $Q \in F_2$.
\end{lem}
\begin{proof}
Since $X$ is finite, we may let $A$ be a member of $F_2 \setminus F_1$ such that no subset of $A$ is in $F_2 \setminus F_1$. Then for all $B \in F_1$, $A \cap B$ either equals $A$ or is in $F_1$, so $A$ is a quasi-closed set of $\langle X, \phi_1 \rangle$. \newline
\end{proof}

\noindent
The next lemma establishes a connection between quasi-closed sets and the right sides of implications.
\begin{lem}\label{QuasBasis}
Let $X$ be a finite set, let $\langle X, \phi \rangle$ be a closure system on $X$, and let $F$ be the family of closed sets associated with $\langle X, \phi \rangle$. Let $\Sigma$ be a set of implications on $X$. Let $F_\Sigma$ be the family of closed sets associated with $\Sigma$. Then $F_\Sigma = F$ if and only if the following two conditions hold:
\begin{enumerate}
    \item For every implication $(A \rightarrow B) \in \Sigma$, $B \subseteq \phi(A)$.
    \item For every quasi closed set $Q$ associated with $\langle X, \phi \rangle$, there exists an implication $(A \rightarrow B) \in \Sigma$ such that $A \subseteq Q$ and $B \not \subseteq Q$.
\end{enumerate}
\end{lem}
\begin{proof}
If condition (1) holds, then for any $S \in F$ and any $(A \rightarrow B) \in \Sigma$, $A \subseteq S \implies \phi(A) \subseteq S \implies B \subseteq S$, so $S \in F_\Sigma$; then it follows that $F \subseteq F_\Sigma$. Condition (2) implies that for any quasi-closed set $Q$ associated with $F$, $Q \not \in F_\Sigma$. Then by Lemma \ref{QuasInc}, $F_\Sigma$ is not a proper super set of $F$, so $F_\Sigma = F$.

The reverse direction is straightforward. If condition (1) fails and $A \rightarrow B$ is an implication in $\Sigma$ such that $B \not \subseteq \phi(A)$, then $\phi(A) \not \in F_\Sigma$ (and $\phi(A) \in F$). If condition (2) fails and $Q$ is a quasi-closed set associated with $\langle X, \phi \rangle$ such that for all $(A \rightarrow B) \in \Sigma$, $A \subseteq Q \implies B \subseteq Q$, then $Q \in F_\Sigma$ (and $Q \not \in F$).
\newline
\end{proof}

\noindent
The following theorem establishes a form of independence between implications corresponding to distinct essential sets.
\begin{thm} \label{MixBasis}
Let $\langle X, \phi \rangle$ be a finite closure system, let $E_1, ..., E_n$ be the essential sets of $\langle X, \phi \rangle$, and let $\Sigma_1, ..., \Sigma_n$ be implication bases for $\langle X, \phi \rangle$. For each $i$, let $\Sigma_i' = \{(A \rightarrow B) \in \Sigma_i : \phi(A) = E_i\}$. Let $\Sigma = \cup \Sigma_i'$. Then $\Sigma$ is a valid implication basis of $\langle X, \phi \rangle$.
\end{thm}
\begin{proof}
Since every implication in $\Sigma$ is included in another valid implication basis of $\langle X, \phi \rangle$, $\Sigma$ satisfies condition (1) of Lemma \ref{QuasBasis}. Let $Q$ be a quasi closed set associated with $\langle X, \phi \rangle$, and let $E_i = \phi(Q)$. By Lemma \ref{QuasBasis}, we may let $A \rightarrow B$ be an implication in $\Sigma_i$ such that $A \subseteq Q$ and $B \not \subseteq Q$. Then $\phi(A) \not \subseteq Q$, and since $Q$ is quasi-closed it follows that $\phi(A) = \phi(Q) = E_i$. Therefore $(A \rightarrow B) \in \Sigma$. It now follows that $\Sigma$ satisfies condition (2) of Lemma \ref{QuasBasis}, so $\Sigma$ is a valid implication basis for $\langle X, \phi \rangle$. \newline
\end{proof}

\noindent
The following corollary resolves Conjecture 67 of \cite{AdN14}.
\begin{cor} \label{OptRight}
Let $\langle X, \phi \rangle$ be a finite closure system, let $E$ be an essential set of $\langle X, \phi \rangle$, and let $\Sigma$ be an optimal basis for $\langle X, \phi \rangle$. If $A_1 \rightarrow B_1, ..., A_n \rightarrow B_n$ are all the implications in $\Sigma$ where $\phi(A_i)=E$, then $s=|B_1|+...+|B_n|$ is fixed (i.e., $s$ does not depend on the choice of $\Sigma$).
\end{cor}
\begin{proof}
Let $\sigma$ be the saturation operator associated with $\langle X, \phi \rangle$. Let $\Sigma'$ be another optimal basis for $\langle X, \phi \rangle$ where $A_1' \rightarrow B_1', ..., A_n' \rightarrow B_n'$ are all the implications in $\Sigma'$ with $\phi(A_i')=E$ and $\sigma(A_i')=\sigma(A_i)$. For each $i$, $|A_i'|=|A_i|$ because $A_i$ and $A_i'$ have minimal cardinality among all sets with saturation equal to $\sigma(A_i)$ \cite[Theorem 5 (c)]{Wild94}. Let $s' = |B_1'| + ... + |B_n'|$ and assume without loss of generality that $s' \leq s$. Let $\Sigma''=\{A \rightarrow B \mid (A \rightarrow B) \in \Sigma' \land \phi(A) = E\} \cup \{A \rightarrow B \mid (A \rightarrow B) \in \Sigma \land \phi(A) \neq E\}$. By Theorem \ref{MixBasis}, $\Sigma''$ is a valid basis for $\langle X, \phi \rangle$. By construction, the size (i.e. the sum of all cardinalities of the left and right sides each implication) of $\Sigma''$ is no greater than the size of $\Sigma$. But $\Sigma$ is optimal, so the size of $\Sigma''$ must equal the size of $\Sigma$, and it follows that $s' = s$.
\end{proof}

\end{document}